\documentclass[12pt]{amsart}
\usepackage{}

\usepackage{amsmath}
\usepackage{amsfonts}
\usepackage{amssymb}
\usepackage[all]{xy}           

\usepackage{bbding}
\usepackage{txfonts}
\usepackage{amscd}

\usepackage[shortlabels]{enumitem}
\usepackage{ifpdf}
\ifpdf
  \usepackage[colorlinks,final,backref=page,hyperindex]{hyperref}
\else
  \usepackage[colorlinks,final,backref=page,hyperindex,hypertex]{hyperref}
\fi
\usepackage{tikz}
\usepackage[active]{srcltx}
\allowdisplaybreaks
\makeatletter
\newtheorem{thm}{Theorem}[section]
\newtheorem{lem}[thm]{Lemma}
\newtheorem{cor}[thm]{Corollary}
\newtheorem{pro}[thm]{Proposition}

\newtheorem{rmk}[thm]{Remark}
\newtheorem{defi}[thm]{Definition}

\newcommand{\be }{\begin{equation}}
\newcommand{\ee }{\end{equation}}

\newcommand{\pf}{\noindent{\bf Proof.}\ }
\newcommand {\emptycomment}[1]{} 
\newcommand{\g}{\frkg}

\newcommand{\B}{\mathsf{B}}

\newcommand{\huaG}{\mathcal{G}}
\newcommand{\adj}{\frka\frkd}

\newcommand{\frka}{\mathfrak a}

\newcommand{\frkd}{\mathfrak d}

\newcommand{\frkf}{\mathfrak f}
\newcommand{\frkg}{\mathfrak g}

\def\qed{\hfill ~\vrule height6pt width6pt depth0pt}

\newcommand{\half}{\frac{1}{2}}

\newcommand{\Dorfman}[1]{\left\llbracket  #1\right\rrbracket }
\newcommand{\Courant}[1]{\{ #1\}}

\newcommand{\id}{\mathrm{id}}

\newcommand{\br}[1]{   [ \cdot,    \cdot  ]   }

\newcommand{\Hom}{\mathrm{Hom}}

\newcommand{\gl}{\mathfrak {gl}}
\newcommand{\ol}{\mathfrak {ol}}

\newcommand{\ad}{\mathrm{ad}}
\newcommand{\pr}{\mathrm{pr}}

\newcommand{\img}{\mathrm{im}}

\newcommand{\sgn}{\mathrm{sgn}}

\begin{document}
\title[Omni-representations of Leibniz algebras]
{Omni-representations of Leibniz algebras}

\author{Zhangju Liu}
\address{Department of Mathematics, Peking University,  Beijing 100871,
China}
\email{liuzj@pku.edu.cn}
\author{Yunhe Sheng}
\address{Department of Mathematics, Jilin University, Changchun 130012, Jilin, China}
\email{shengyh@jlu.edu.cn}

\date{}

\begin{abstract}
In this paper, first we
introduce the notion of an omni-representation  of a Leibniz
algebra $\g$ on a vector space $V$ as a Leibniz algebra homomorphism from $\g$ to the omni-Lie algebra $\gl(V)\oplus V$.   Then we introduce the  omni-cohomology theory associated to omni-representations and establish  the relation between omni-cohomology groups   and Loday-Pirashvili cohomology groups.
\end{abstract}



\maketitle
\tableofcontents

\emptycomment{\begin{cor}
 With the above notations, $J$ is totally skew-symmetric.
\end{cor}
\pf For all $x,z\in\g$, by \eqref{eq:center}, we have
\begin{eqnarray*}
  J_{x,x,z}=\frac{1}{4}\big([[z,x]_\g,x]_\g+[[x,z]_\g,x]_\g+[[x,x]_\g,z]_\g\big)=0.
\end{eqnarray*}
Similarly, we have $J_{x,z,x}=0$ and $J_{z,x,x}=0$. Thus, $J$ is totally skew-symmetric.
\qed}
\section{Introduction}
Leibniz algebras were first discovered by Bloh who called them D-algebras  \cite{Bloh}. Then Loday rediscovered this algebraic structure and called them Leibniz algebras \cite{Loday and Pirashvili}.
A Leibniz algebra is a vector space
$\frkg$, endowed with a linear map
$[\cdot,\cdot]_\frkg:\frkg\otimes\frkg\longrightarrow\frkg$
satisfying
\begin{equation}\label{eq:Leibniz}
[x,[y,z]_\frkg]_\frkg=[[x,y]_\frkg,z]_\frkg+[y,[x,z]_\frkg]_\frkg,\quad
\forall~x,y,z\in \frkg.
\end{equation}
In particular, if the bracket $[\cdot,\cdot]_\frkg$ is skew-symmetric, then it is a Lie algebra. Leibniz algebras have important applications in both mathematics and mathematical physics, e.g. the section space of a Courant algebroid is a Leibniz algebra \cite{ROy} and the underlying algebraic structure of an embedding tensor is also a Leibniz algebra which further leads to  applications in higher gauge theories \cite{KS}.

The theory of  representations  of Leibniz algebras
was introduced and studied  in \cite{Loday and Pirashvili}.

\begin{defi} \label{defi:rep old}
A representation of a Leibniz algebra
$(\frkg,[\cdot,\cdot]_\frkg)$ is a triple  $(V,l,r)$, where  $V$ is
a vector space equipped with two linear maps
$l:\g\longrightarrow\gl(V)$ and $r:\g\longrightarrow\gl(V)$ such
that the following equalities hold:
\begin{equation}\label{condition of rep}
l_{[x,y]_\frkg}=[l_{x},l_{y}],\quad
r_{[x,y]_\frkg}=[l_{x},r_{y}],\quad r_{y}\circ l_{x}=-r_{y}\circ
r_{x}, ~~~ \, ~~~\, ~~~  \forall x,y\in\frkg.
\end{equation}
\end{defi}
Especially, faithful representations of Leibniz algebras were studied by Barnes in \cite{faithful rep}; conformal representations of Leibniz algebras were studied by Kolesnikov in \cite{conformal rep}; dual representations
of Leibniz algebras were given in  \cite{TS} in their study of Leibniz bialgbras. Representations of symmetric Leibniz algebras were studied by Benayadi in \cite{Bena}.

Note that a representation of a Lie algebra $\g$ on a vector space
$V$ is a Lie algebra homomorphism from $\g$ to the Lie algebra
$\gl(V)$, which realizes an abstract Lie algebra as a subalgebra of
the general linear Lie algebra. While the above representation of a Leibniz algebra does not have this advantage. The purpose of this paper is to introduce a new representation theory so that it can realize an abstract Leibniz algebra as a subalgebra of a concrete Leibniz algebra. The omni-Lie algebra $\gl(V)\oplus V$ introduced by Weinstein in \cite{Weinomni} is naturally a Leibniz algebra and the main ingredient in our study. We introduce the notion of an omni-representation of a Leibniz algebra $\g$ on a vector space $V$ which is a homomorphism from $\g$ to the Leibniz algebra  $\gl(V)\oplus V$. We show that a usual representation  $(V,l,r)$ gives rise to an omni-representation $\rho=(l^*\otimes 1+ 1\otimes l)+r $    of $\g$ on  $V^*\otimes V$.

The  cohomology theory of Leibniz algebras
was also developed by Loday and   Pirashvili  in \cite{Loday and Pirashvili}.
See   \cite{ALO,Omirov,Cheng,Wagemann,FMM,HuPeiLiu} for
more applications of Loday-Pirashvili cohomologies of Leibniz algebras. We also develop the omni-cohomology theory for   omni-representations introduced above, and give the relation between omni-cohomology groups   and Loday-Pirashvili cohomology groups.

The paper is organized as follows. In Section \ref{sec:rep}, we restudy representation of Leibniz algebras and the corresponding semidirect products. In Section \ref{sec:naive}, we introduce the notion of omni-representations of Leibniz algebras and study the relation between omni-representations and the usual representations. In Section \ref{sec:coh}, we introduce omni-cohomology groups for Leibniz algebras with coefficients in omni-representations, and establish the relation between omni-cohomology groups and Loday-Pirashvili cohomology groups.

\section{Representations of Leibniz algebras}\label{sec:rep}
Let $(V,l,r)$ be a representation of a Leibniz algebra
$(\frkg,[\cdot,\cdot]_\frkg)$.

\begin{defi} \label{defi:cohomo old}
The  Loday-Pirashvili cohomology of $\frkg$ with coefficients in $V$ is the
cohomology of the cochain complex
$C^k(\frkg,V)=\Hom(\otimes^k\frkg,V), (k\geq0)$ with the coboundary
operator
 $$\partial:C^k(\frkg,V)\longrightarrow C^{k+1}(\frkg,V)$$
defined by
\begin{eqnarray}
\nonumber\partial
c^k(x_1,\dots,x_{k+1})&=&\sum_{i=1}^k(-1)^{i+1}l_{x_i}(c^k(x_1,\dots,\widehat{x_i},\dots,x_{k+1}))\\
\nonumber&&+(-1)^{k+1}r_{x_{k+1}}(c^k(x_1,\dots,x_k))\\
\label{formulapartial}&&+\sum_{1\leq i<j\leq
k+1}(-1)^ic^k(x_1,\dots,\widehat{x_i},\dots,x_{j-1},[x_i,x_j]_\frkg,x_{j+1},\dots,x_{k+1}).
\end{eqnarray}
The resulting cohomology is denoted by $H^\bullet(\g;l,r)$.
\end{defi}

Obviously, $(\mathbb R, 0, 0)$ is a representation of  a Leibniz algebra
$(\frkg,[\cdot,\cdot]_\frkg)$, which is
called the trivial representation. Denote the corresponding cohomology
by $H^\bullet(\g).$ Another important representation is the adjoint
representation  $(\g, \ad_L, \ad_R)$, where
$\ad_L:\g\longrightarrow\gl(\g)$ and
$\ad_R:\g\longrightarrow\gl(\g)$  are defined as follows:
\begin{equation}
  \ad_L(x)(y)=[x,y]_\g,\quad \ad_R(x)(y)=[y,x]_\g, ~~~ \, ~~~\, ~~~  \forall x,y\in\frkg.
\end{equation}
 The corresponding cohomology is denoted by
 $H^\bullet(\g;\ad_L, \ad_R).$

 A permutation $\sigma\in\mathbb S_n$ is called an $(i,n-i)$-{\bf shuffle} if $\sigma(1)<\cdots <\sigma(i)$ and $\sigma(i+1)<\cdots <\sigma(n)$. If $i=0$ or $n$ we assume $\sigma=\id$. The set of all $(i,n-i)$-shuffles will be denoted by $\mathbb S_{(i,n-i)}$. The notion of an $(i_1,\cdots,i_k)$-shuffle and the set $\mathbb S_{(i_1,\cdots,i_k)}$ are defined analogously.

Let $\g$ be a vector space. We consider the graded vector space $C^*(\g,\g)=\oplus_{n\ge 1}C^n(\g,\g)=\oplus_{n\ge 1}\Hom(\otimes^n\g,\g)$. The Balavoine bracket on the graded vector space $C^*(\g,\g)$ is given
by:
\begin{eqnarray}\label{leibniz-bracket}
[P,Q]_\B=P\bar{\circ}Q-(-1)^{pq}Q\bar{\circ}P,
\end{eqnarray}
for all $P\in C^{p+1}(\g,\g),Q\in C^{q+1}(\g,\g),$ where $P\bar{\circ}Q\in C^{p+q+1}(\g,\g)$ is defined by
\begin{eqnarray}
P\bar{\circ}Q=\sum_{k=1}^{p+1}P\circ_k Q,
\end{eqnarray}
and $\circ_k$ is defined by
\begin{eqnarray}
&&(P\circ_kQ)(x_1,\cdots,x_{p+q+1})
=\sum_{\sigma\in\mathbb S_{(k-1,q)}}(-1)^{\sigma}(-1)^{(k-1)q}\\ \nonumber&&\qquad P(x_{\sigma(1)},\cdots,x_{\sigma(k-1)},Q(x_{\sigma(k)},\cdots,x_{\sigma(k+q-1)},x_{k+q}),x_{k+q+1},\cdots,x_{p+q+1}).
\end{eqnarray}
It is well known that

\begin{thm}{\rm (\cite{Balavoine,Fialowski})}\label{leibniz-algebra-B}
With the above notations, $(C^*(\g,\g),[\cdot,\cdot]_{\B})$ is a graded Lie algebra.
\end{thm}
In particular, for
$\alpha\in C^2(\g,\g)$, we have
\begin{equation}
 ~[\alpha,\alpha]_\B(x,y,z)=2\alpha\circ\alpha(x,y,z)=2\Big(\alpha(\alpha(x,y),z)-\alpha(x,\alpha(y,z))+\alpha(y,\alpha(x,z))\Big).
\end{equation}
Thus, $\alpha$ defines a Leibniz algebra structure if and only if $[\alpha,\alpha]_\B=0.$

\emptycomment{
The graded vector space $\bigoplus_k C^k(\frkg,\g)$ equipped with
the graded bracket
\begin{equation}
  [\alpha,\beta]=\alpha\circ\beta+(-1)^{pq+1}\beta\circ\alpha, \quad\forall ~\alpha\in C^{p+1}(\frkg,\g),~\beta\in C^{q+1}(\frkg,\g)
\end{equation}
  is a   graded Lie algebra, where $\alpha\circ\beta\in C^{p+q+1}(\g,\g)$ is defined by
\begin{eqnarray*}
  \alpha\circ\beta(x_1,\cdots,x_{p+q+1})&=&\sum_{k=0}^{p}(-1)^{kq}\Big(\sum_{\sigma\in sh(k,q)}\sgn(\sigma)\alpha(x_{\sigma(1)},\cdots,x_{\sigma(k)},\\
  &&\qquad\beta(x_{\sigma(k+1)},\cdots,x_{\sigma(k+q)},x_{k+q+1}),x_{k+q+2},\cdots,x_{p+q+1})\Big).
\end{eqnarray*}
See \cite{Balavoine,Fialowski} for more details.
}

For a representation  $(V,l,r)$ of $\g$, it is obvious that
$(V,l,0)$ is also a representation of $\g$. Thus,  we have two
semidirect product Leibniz algebras
 $\g\ltimes_{(l,r)}V$  and $\g\ltimes_{(l,0)}V$ with the brackets $[\cdot,\cdot]_{(l,r)}$ and
$[\cdot,\cdot]_{(l,0)}$ given respectively by
\begin{eqnarray*}
 ~ [x+u,y+v]_{(l,r)}&=&[x,y]_\g+l_xv+r_yu,\\
  ~ [x+u,y+v]_{(l,0)}&=&[x,y]_\g+l_xv.
\end{eqnarray*}
The right action $r$ induces a linear map $\overline{r}:(\g\oplus
V)\otimes (\g\oplus V)\longrightarrow \g\oplus V$ as follows:
$$
 \overline{r}(x+u,y+v)=r_yu,\quad \forall x,y\in\g,~u,v\in V.
$$

The adjoint representations $\ad_L$ and $\ad_R$ of the Leibniz algebra
$\g\ltimes_{(l,0)}V$ are given by
  $$\ad_L(x+u)(y+v)=[x,y]_\g+l_xv, ~~~ \, ~~~
  \ad_R(x+u)(y+v)=[y,x]_\g+l_yu.$$
\begin{thm}
  With the above notations, $\overline{r}$ satisfies the following Maurer-Cartan   equation on the Leibniz algebra $\g\ltimes_{(l,0)}V$:
  $$
  \partial \overline{r}-\half[\overline{r},\overline{r}]_\B=0,
  $$
   where $\partial$ is  the coboundary
operator for the Leibniz algebra  $\g\ltimes_{(l,0)}V$ with coefficients in the adjoint representation.
 Consequently, the Leibniz algebra $\g\ltimes_{(l,r)}V$ is a deformation of the Leibniz algebra $\g\ltimes_{(l,0)}V$
  via the  Maurer-Cartan element $\overline{r}$.
  \end{thm}
\begin{proof} By direct computation, we have
\begin{eqnarray*}
\partial \overline{r}(x+u,y+v,z+w)&=&\ad_L(x+u)\overline{r}(y+v,z+w)-\ad_L(y+v)\overline{r}(x+u,z+w)\\
&&-\ad_R(z+w)\overline{r}(x+u,y+v)- \overline{r}([x+u,y+v]_{(l,0)},z+w)\\
 &&+\overline{r}(x+u,[y+v,z+w]_{(l,0)})-\overline{r}(y+v,[ x+u,z+w]_{(l,0)})\\
  &=&l_xr_zv-l_yr_zu-r_zl_xv+r_{[y,z]_\g}u-r_{[x,z]_\g}v.
\end{eqnarray*}
On the other hand, we have
\begin{eqnarray*}
  [\overline{r},\overline{r}]_\B(x+u,y+v,z+w)&=&2\big(\overline{r}(\overline{r}(x+u,y+v),z+w)-\overline{r}(x+u,\overline{r}(y+v,z+w))\\&&+\overline{r}(y+v,\overline{r}(x+u,z+w))\big)
  \\&=&2r_zr_yu.
\end{eqnarray*}
Thus, by \eqref{condition of rep}, we have
\begin{eqnarray*}
  \Big(\partial \overline{r}-\half[\overline{r},\overline{r}]_\B\Big)(x+u,y+v,z+w)&=&l_xr_zv-l_yr_zu-r_zl_xv+r_{[y,z]_\g}u-r_{[x,z]_\g}v-r_zr_yu\\
  &=&l_xr_zv-l_yr_zu-r_zl_xv+r_{[y,z]_\g}u-r_{[x,z]_\g}v+r_zl_yu\\
  &=&0.
\end{eqnarray*}
The proof is finished. \end{proof}

 Define $l^*:\g\longrightarrow\gl(V^*)$  by
  $$\langle l^*_x(\xi),u\rangle=-\langle\xi,l_xu\rangle,  ~~~~~ \, ~~~ \forall x \in \g, ~ \xi \in V^*, ~ u\in V.$$
  It is straightforward to see that $(V^*\otimes V,l^*\otimes 1+ 1\otimes l,0)$ is also a representation of $\g$,
  where $l^*\otimes 1+ 1\otimes l:\g\longrightarrow\gl(V^*\otimes V)$ is given by
$$
(l^*\otimes 1+ 1\otimes l)_x(\xi\otimes u)=(l^*_x\xi)\otimes u+\xi\otimes l_xu,\quad \forall \xi\in V^*,~u\in V.
$$

Since  $V^*\otimes V \cong \gl(V)$, an element $\xi\otimes u$ in
$V^*\otimes V$ can be identified with a linear map $A\in\gl(V)$ via
$A(v)=\langle\xi,v\rangle u$.

\begin{pro}\label{pro:1cocycle}
  With the above notations,  for all $A\in V^*\otimes V \cong \gl(V)$, we have
  $$
  (l^*\otimes 1+ 1\otimes l)_x A=[l_x,A]=l_x\circ A-A\circ l_x.
  $$
  Moreover, the right action  $r:\g\longrightarrow \gl(V)$ is a $1$-cocycle on $\g$ with coefficients in  the
   representation $(V^*\otimes V,l^*\otimes 1+ 1\otimes l,0)$.
\end{pro}
\begin{proof} Write $A=\xi\otimes u$, then we have
\begin{eqnarray*}
  (l^*\otimes 1+ 1\otimes l)_x A(v)&=&(l^*\otimes 1+ 1\otimes l)_x (\xi\otimes u)(v)=\Big((l^*_x\xi)\otimes u+\xi\otimes l_xu\Big)(v)\\
  &=&\langle l^*_x\xi,v\rangle u+\langle\xi,v\rangle l_xu=-\langle\xi,l_xv\rangle u+\langle\xi,v\rangle l_xu\\
  &=&[l_x,A](v).
\end{eqnarray*}
Therefore,  by \eqref{condition of rep}, we have
\begin{eqnarray*}
  \partial r(x,y)= (l^*\otimes 1+ 1\otimes l)_xr(y)-r([x,y]_\g)=[l_x,r_y]-r_{[x,y]_\g}=0,
\end{eqnarray*}
which implies that $r$ is a $1$-cocycle.\end{proof}

\section{Omni-representations of Leibniz algebras}\label{sec:naive}

It is known that the aim of a representation is to  realize an
abstract algebraic structure  as a class of linear transformations
on a vector space.  Such as  a Lie algebra  representation is a
homomorphism from $\g$ to the general linear Lie algebra $\gl(V)$.
Unfortunately, the  representation of a Leibniz algebra discussed
above does not realize a Leibniz algebra as a subalgebra of certain explicit Leibniz algebra. Therefore, it is reasonable for us to  provide an
alternative definition for the representation of a Leibniz algebra.
It is lucky that there is a god-given Leibniz algebra  worked as a
``general linear algebra''  defined as follows:

Given a vector space $V$,   then $(V,l=\id,r=0)$ is a natural
representation of $\gl(V)$, which is viewed as a Leibniz algebra.
 The corresponding semidirect product
Leibniz algebra structure on $\gl(V)\oplus V$ is given by
\begin{eqnarray*}
\Courant{A+u,B+v}&=&[A,B]+Av,\quad \forall~A,B\in \gl(V),~u,v\in V.
\end{eqnarray*}
This Leibniz algebra is called an {\bf
omni-Lie algebra} and denoted by $\ol(V)$.
 The notion of an omni-Lie algebra was introduced  by Weinstein in \cite{Weinomni}
 as the linearization of a Courant algebroid. The notion of a Courant algebroid was introduced
 in \cite{LWXmani}, which has been widely  applied in many
fields both for mathematics and physics (see \cite{KS} for
more details). Its Leibniz algebra structure also  played  an
important role when studying the
 integrability of Courant brackets \cite{kinyon-weinstein,LW}.

 Notice that  the skew-symmetric
 bracket,
 \begin{equation}
   \Dorfman{A+u,B+v}=[A,B]+\frac{1}{2}(Av-Bu),
 \end{equation}
which is obtained via the skew-symmetrization of $\Courant{\cdot,\cdot}$, is used
 in Weinstein's original
 definition. \emptycomment{As a special case in Theorem \ref{thm:main1},  $(\gl(V)\oplus
 V,\Dorfman{\cdot,\cdot})$ is a Lie 2-algebra.}
Even though an  omni-Lie algebra is  not a Lie algebra,
  all Lie algebra structures on $V$ can be characterized by the Dirac structures
  in   $\ol(V)$. In fact, the next proposition will show that  every Leibniz algebra structure on $V$ can
  be realized as a Leibniz subalgebra of $\ol(V)$.
For any $\varphi:V\longrightarrow\gl(V)$, consider its graph
$$\huaG_\varphi=\{\varphi(u)+u\in\gl(V)\oplus V|~\forall u\in V\}.$$

\begin{pro}
  With the above notations, $\huaG_\varphi$ is a Leibniz subalgebra  of $\ol(V)$ if and only if
  \begin{equation}\label{eq:gr}
    [\varphi(u),\varphi(v)]=\varphi(\varphi(u)v),\quad\forall~u,v\in V.
  \end{equation}
  Furthermore, under this condition, $(V,[\cdot,\cdot]_\varphi)$ is a Leibniz algebra, where the linear map $[\cdot,\cdot]_\varphi:V\otimes V\longrightarrow V$ is given by
  \begin{equation}
    [u,v]_\varphi=\varphi(u)v,\quad \forall~u,v\in V.
  \end{equation}
\end{pro}
\begin{proof} Since $\ol(V)$ is a Leibniz algebra, we only need to show that $\huaG_\varphi$ is closed if and only if \eqref{eq:gr} holds. The conclusion follows from
\begin{eqnarray*}
  \Courant{\varphi(u)+u,\varphi(v)+v}=[\varphi(u),\varphi(v)]+\varphi(u)v.
\end{eqnarray*}
The other conclusion is straightforward. The proof is finished. \end{proof}

\begin{rmk}
 The condition \eqref{eq:gr} actually means that $\varphi$ is an embedding tensor. See \cite{STZ} for more details about embedding tensors.
\end{rmk}

Recall that a representation of a Lie algebra $\g$ on a vector space
$V$ is a Lie algebra homomorphism from $\g$ to the Lie algebra
$\gl(V)$, which realizes an abstract Lie algebra as a subalgebra of
the general linear Lie algebra. Similarly, For a  Leibniz algebra, we
suggest the following definition:
\begin{defi}
  An {\bf omni-representation} of a Leibniz algebra
$(\frkg,[\cdot,\cdot]_\frkg)$ on a vector space $V$ is a Leibniz algebra homomorphism $\rho:\g\longrightarrow\ol(V)$.
\end{defi}

According to the two components of   $\gl(V)\oplus V$,  every linear
map $\rho:\g\longrightarrow\ol(V)$   splits to two linear
maps: $\phi:\g\longrightarrow \gl(V)$ and
   $\theta:\g\longrightarrow V$. Then, we have


\begin{pro}\label{pro:rep con}
 A linear map $\rho=\phi+\theta:\g\longrightarrow\gl(V)\oplus V$ is an omni-representation  of  a Leibniz algebra
$(\frkg,[\cdot,\cdot]_\frkg)$
 if and only if $(V,\phi,0) $ is a representation of the Leibniz algebra
$(\frkg,[\cdot,\cdot]_\frkg)$ and $\theta:\g\longrightarrow V$ is a
 $1$-cocycle on the Leibniz algebra
$(\frkg,[\cdot,\cdot]_\frkg)$ with coefficients in
the representation $(V,\phi,0)$.
\end{pro}
\begin{proof} On one hand, we have
$$
\rho([x,y]_\g)=\phi([x,y]_\g)+\theta([x,y]_\g).
$$
On the other hand, we have
\begin{eqnarray*}
  \Courant{\rho(x),\rho(y)}= \Courant{\phi(x)+\theta(x),\phi(y)+\theta(y)}=[\phi(x),\phi(y)]+\phi(x)\theta(y).
\end{eqnarray*}
Thus, $\rho$ is a homomorphism if and only if
\begin{eqnarray}\label{eq:rep con1}
   \phi([x,y]_\g)&=&[\phi(x),\phi(y)],\\
   \label{eq:rep con2}\theta([x,y]_\g)&=&\phi(x)\theta(y).
 \end{eqnarray}
The equality \eqref{eq:rep con1} implies that $(V,\phi,0) $ is a representation of the Leibniz algebra
$(\frkg,[\cdot,\cdot]_\frkg)$, and the equality \eqref{eq:rep
con2} implies that
$\theta:\g\longrightarrow V$ is a $1$-cocycle on the Leibniz algebra
$(\frkg,[\cdot,\cdot]_\frkg)$ with coefficients in
the representation $(V,\phi,0)$. \end{proof}

A usual representation in the sense of Definition \ref{defi:rep old} gives rise to an omni-representation naturally.

\begin{thm}
  Let $(V,l,r)$ be a representation of  a Leibniz algebra
$(\frkg,[\cdot,\cdot]_\frkg)$. Then
  $$\rho=(l^*\otimes 1+ 1\otimes l)+r: \g \longrightarrow \ol(V^*\otimes V) $$ is an omni-representation  of $\g$ on  $V^*\otimes V$.
\end{thm}

\begin{proof} By Proposition \ref{pro:1cocycle}, $r:\g\longrightarrow \gl(V)$ is a $1$-cocycle on $\g$ with  coefficients in the representation $(V^*\otimes V,l^*\otimes 1+ 1\otimes l,0)$.
By Proposition \ref{pro:rep con}, $\rho=(l^*\otimes 1+ 1\otimes l)+r$ is a homomorphism from $\g$ to $\ol(V^*\otimes V)$, which implies that $\rho$ is an omni-representation  of $\g$ on  $V^*\otimes V$.\end{proof}

A {\bf trivial  omni-representation}  of a Leibniz algebra
$(\frkg,[\cdot,\cdot]_\frkg)$ on $\mathbb R$
is defined
 to be a homomorphism $$\rho_T=\phi+\theta:\g\longrightarrow\gl(\mathbb R)
 \oplus \mathbb R $$
 such that $\phi=0$.  By Proposition \ref{pro:rep con}, we have
\begin{pro}\label{trivial}
  Trivial   omni-representations of a Leibniz algebra
$(\frkg,[\cdot,\cdot]_\frkg)$ are in one-to-one correspondence
   to $\xi\in\g^*$ such that $\xi|_{[\g,\g]_\g}=0$.
\end{pro}

The {\bf adjoint   omni-representation} $\adj$ of a Leibniz algebra
$(\frkg,[\cdot,\cdot]_\frkg)$ on $\g$ is defined to be the
homomorphism $$\adj=\ad_L+\id: \g\longrightarrow\gl(\g)\oplus \g.$$

\section{Omni-cohomologies of Leibniz algebras}\label{sec:coh}

In this section, we introduce omni-cohomologies of Leibniz algebras associated to omni-representations, and show that omni-cohomology groups   and Loday-Pirashvili   cohomology groups are isomorphic for the trivial representations and adjoint representations.

  Let $\rho:\g\longrightarrow\ol(V)$ be  an omni-representation of   a Leibniz algebra
$(\frkg,[\cdot,\cdot]_\frkg)$. It is obvious that  $\img(\rho) \subset\ol(V)$ is a Leibniz subalgebra so that one can
 define the
set of $k$-cochains by
$$
C^k(\g;\rho)=\{f:\otimes ^k\g\longrightarrow \img(\rho)\}$$ and
 an operator $\delta:C^k(\g;\rho)\longrightarrow
C^{k+1}(\g;\rho)$ by
\begin{eqnarray}
\nonumber\delta
c^k(x_1,\dots,x_{k+1})&=&\sum_{i=1}^k(-1)^{i+1}\Courant{\rho(x_i),c^k(x_1,\dots,\widehat{x_i},\dots,x_{k+1})}
\\\nonumber&&+(-1)^{k+1}\Courant{ c^k(x_1,\dots,x_k),\rho(x_{k+1})}\\
\label{formulapartial}&&+\sum_{1\leq i<j\leq
k+1}(-1)^ic^k(x_1,\dots,\widehat{x_i},\dots,x_{j-1},[x_i,x_j]_\frkg,x_{j+1},\dots,x_{k+1}).
\end{eqnarray}

\begin{lem}
  With the above notations, we have $\delta^2=0$.
\end{lem}

\begin{proof} For all $x\in\g$ and $u\in\img(\rho)$, define
$$
l_x(u)=\Courant{\rho(x),u},\quad r_x(u)=\Courant{u,\rho(x)}.
$$
By the fact that $\ol(V)$ is a Leibniz algebra, we can deduce that
$(\img(\rho);l,r)$ is a representation of $\g$ on $\img(\rho)$ in
the sense of Definition \ref{defi:rep old}, and $\delta$ is just the
usual coboundary operator for this representation so that
$\delta^2=0.$ \end{proof}

Thus, we have a well-defined cochain complex
$(C^{\bullet}(\g;\rho),\delta).$  The corresponding cohomology is  called the {\bf omni-cohomology} of the Leibniz algebra $(\frkg,[\cdot,\cdot]_\frkg)$ with coefficients in the omni-representation $\rho$, and denoted by
$H_{omni}^\bullet(\g;\rho)$. In particular, $H_{omni}^\bullet(\g)$
and $H_{omni}^\bullet(\g;\adj)$ denote the  omni-cohomologies
with coefficients in  trivial    omni-representation and adjoint   omni-representation of $\g$ respectively.

\begin{thm}
 With the above notations,   we have  $H_{omni}^\bullet(\g) =H^\bullet(\g)$.
\end{thm}
\begin{proof} If $[\g,\g]_\g=\g$, there is only one trivial omni-representation $\rho=0$
 by Proposition \ref{trivial}.
 In this case, all the cochains are also $0$. Thus, $H_{omni}^\bullet(\g)=0$. On the other hand, under the condition  $[\g,\g]_\g=\g$, it is straightforward to deduce that for any $\xi\in C^k(\g)$, $\partial \xi=0$ if and only if $\xi=0$. Thus, $H^\bullet(\g)=0$.

If $[\g,\g]_\g\neq\g$, any $0\neq\xi\in\g^*$ such that
$\xi|_{[\g,\g]_\g}=0$ gives rise to a trivial   omni-representation
$\rho_T$. Furthermore, we have $\img(\rho_T)=\mathbb R$ and
$C^k(\g)=\otimes^k\g^*$. Thus, the sets of cochains are the same
associated to two kinds of representations. Since $V$ is an abelian
subalgebra in $\ol(V)$, for any $\xi\in C^k(\g)$, we have
\begin{eqnarray*}
\delta
\xi(x_1,\dots,x_{k+1})&=&\sum_{i=1}^k(-1)^{i+1}\Courant{\rho_T(x_i),c^k(x_1,\dots,\widehat{x_i},\dots,x_{k+1})}
\\\nonumber&&+(-1)^{k+1}\Courant{ c^k(x_1,\dots,x_k),\rho_T(x_{k+1})}\\
\label{formulapartial}&&+\sum_{1\leq i<j\leq
k+1}(-1)^ic^k(x_1,\dots,\widehat{x_i},\dots,x_{j-1},[x_i,x_j]_\frkg,x_{j+1},\dots,x_{k+1})\\
&=&\sum_{1\leq i<j\leq
k+1}(-1)^ic^k(x_1,\dots,\widehat{x_i},\dots,x_{j-1},[x_i,x_j]_\frkg,x_{j+1},\dots,x_{k+1})\\
&=&\partial \xi(x_1,\dots,x_{k+1}).
\end{eqnarray*}
Thus, we have $H_{omni}^\bullet(\g)=H^\bullet(\g)$.\end{proof}

For the adjoint   omni-representation $\adj=\ad_L+\id:
\g\longrightarrow\gl(\g)\oplus \g,$  any $k$-cochain $f$ is uniquely
determined by a linear map $\frkf:\otimes^k\g\longrightarrow \g$
such that
\begin{equation}\label{eq:correspondence}
f=(\ad_L\circ{\frkf},\frkf):\otimes^k\g\longrightarrow \img(\adj).
\end{equation}

\begin{thm}\label{thm:adjoint}
 With the above notations, we have $H_{omni}^\bullet(\g;\adj)=H^\bullet(\g;\ad_L,\ad_R)$.
\end{thm}

\begin{proof} Since any $k$-cochain $f:\otimes^k\g\longrightarrow \img(\adj)$ is uniquely determined
by a linear map $\frkf:\otimes^k\g\longrightarrow \g$ via \eqref{eq:correspondence}. Thus, there is a
 one-to-one correspondence between the sets of cochains associated to
 the two kinds representations via $f\leftrightsquigarrow\frkf$.
 Furthermore, we have
\begin{eqnarray*}
&&\delta f(x_1,\dots,x_{k+1})\\&=&\sum_{i=1}^k(-1)^{i+1}\Courant{\adj(x_i),f(x_1,\dots,\widehat{x_i},\dots,x_{k+1})}
\\\nonumber&&+(-1)^{k+1}\Courant{ f(x_1,\dots,x_k),\adj(x_{k+1})}\\
\label{formulapartial}&&+\sum_{1\leq i<j\leq
k+1}(-1)^if(x_1,\dots,\widehat{x_i},\dots,x_{j-1},[x_i,x_j]_\frkg,x_{j+1},\dots,x_{k+1})\\
&=&\sum_{i=1}^k(-1)^{i+1}\Courant{\ad_L(x_i)+x_i,\ad_L\frkf(x_1,\dots,\widehat{x_i},\dots,x_{k+1})+\frkf(x_1,\dots,\widehat{x_i},\dots,x_{k+1})}
\\\nonumber&&+(-1)^{k+1}\Courant{ \ad_L\frkf(x_1,\dots,x_k)+\frkf(x_1,\dots,x_k),\ad_L(x_{k+1})+x_{k+1}}\\
\label{formulapartial}&&+\sum_{1\leq i<j\leq
k+1}(-1)^i\Big(\ad_L\frkf(x_1,\dots,\widehat{x_i},\dots,x_{j-1},[x_i,x_j]_\frkg,x_{j+1},\dots,x_{k+1})\\
&&+\frkf(x_1,\dots,\widehat{x_i},\dots,x_{j-1},[x_i,x_j]_\frkg,x_{j+1},\dots,x_{k+1})\Big)\\
&=&\sum_{i=1}^k(-1)^{i+1}\Big([\ad_L(x_i),\ad_L\frkf(x_1,\dots,\widehat{x_i},\dots,x_{k+1})]+\ad_L(x_i)\frkf(x_1,\dots,\widehat{x_i},\dots,x_{k+1})\Big)
\\\nonumber&&+(-1)^{k+1}\Big([ \ad_L\frkf(x_1,\dots,x_k),\ad_L(x_{k+1})]+\ad_L\frkf(x_1,\dots,x_k)x_{k+1}\Big)\\
\label{formulapartial}&&+\sum_{1\leq i<j\leq
k+1}(-1)^i\Big(\ad_L\frkf(x_1,\dots,\widehat{x_i},\dots,x_{j-1},[x_i,x_j]_\frkg,x_{j+1},\dots,x_{k+1})\\
&&+\frkf(x_1,\dots,\widehat{x_i},\dots,x_{j-1},[x_i,x_j]_\frkg,x_{j+1},\dots,x_{k+1})\Big)\\
&=&\sum_{i=1}^k(-1)^{i+1}\Big(\ad_L[x_i,\frkf(x_1,\dots,\widehat{x_i},\dots,x_{k+1})]_\g+[x_i,\frkf(x_1,\dots,\widehat{x_i},\dots,x_{k+1})]_\g\Big)
\\\nonumber&&+(-1)^{k+1}\Big( \ad_L[\frkf(x_1,\dots,x_k),x_{k+1}]_\g+[\frkf(x_1,\dots,x_k),x_{k+1}]_\g\Big)\\
\label{formulapartial}&&+\sum_{1\leq i<j\leq
k+1}(-1)^i\Big(\ad_L\frkf(x_1,\dots,\widehat{x_i},\dots,x_{j-1},[x_i,x_j]_\frkg,x_{j+1},\dots,x_{k+1})\\
&&+\frkf(x_1,\dots,\widehat{x_i},\dots,x_{j-1},[x_i,x_j]_\frkg,x_{j+1},\dots,x_{k+1})\Big)\\
&=&\ad_L\partial \frkf(x_1,\dots,x_{k+1})+\partial \frkf(x_1,\dots,x_{k+1}).
\end{eqnarray*}
Thus, we have $\delta f=0$ if and only if $\partial \frkf=0$. Similarly, we can prove that $f$ is exact if and only if $\frkf$ is exact. Thus, the corresponding cohomology groups are isomorphic.
\end{proof}

At last, we consider an omni-representation $\rho$ such that the
image of $\rho$ is
 contained in the graph $\huaG_\varphi$ for some linear map $\varphi:V\longrightarrow\gl(V)$
 satisfying Eq. \eqref{eq:gr}.
In this case, $\rho$
 is of the form $\rho=\varphi\circ \theta+\theta$, where $\theta:\g\longrightarrow V$ is a linear map. A $k$-cochain $f:\otimes^k\g\longrightarrow \img(\rho)$ is of the form $f=\varphi\circ \frkf+\frkf$, where $\frkf:\otimes^k\g\longrightarrow V$ is a linear map.

 Define left and right actions in the sense of Definition \ref{defi:rep old} by
 \begin{eqnarray}
  \label{eq:l} l_xu&=&\pr\Courant{\rho(x),\varphi(u)+u}=\varphi(\theta(x))u;\\
  \label{eq:r}  r_xu&=&\pr\Courant{\varphi(u)+u,\rho(x)}=\varphi(u)\theta(x),
 \end{eqnarray}
 where $\pr$ is the projection from $\gl(V)\oplus V$ to $V$.  Similar to Theorem \ref{thm:adjoint}, it is easy to
 prove that
\begin{thm}
 Let $\rho$ be an omni-representation such that the image of $\rho$ is contained in the graph $\huaG_\phi$ for some linear map $\varphi:V\longrightarrow\gl(V)$ satisfying Eq.  \eqref{eq:gr}. Then we have
 $$
 H_{omni}^\bullet(\g;\rho)=H^\bullet(\g;l,r),
 $$
 where $l$ and $r$ are given by \eqref{eq:l} and \eqref{eq:r} respectively.
\end{thm}

\end{document}